\documentclass[11pt]{article}
\usepackage[utf8]{inputenc}
\usepackage{amsmath,amsfonts,amssymb}
\usepackage{amsthm}
\usepackage{latexsym}
\usepackage[noadjust]{cite}
\usepackage{graphicx}
\usepackage{xcolor}
\usepackage{dirtytalk}
\usepackage{mathtools}
\usepackage[margin=1in]{geometry}
\usepackage{thm-restate}
\usepackage{enumitem}
\usepackage[linesnumbered,ruled]{algorithm2e}
\usepackage[colorlinks=true, allcolors=blue]{hyperref}
\usepackage[nameinlink,capitalise]{cleveref}
\hypersetup{
    citecolor={violet}
}

\newtheorem{theorem}{Theorem}[section]
\newtheorem{corollary}[theorem]{Corollary}
\newtheorem{lemma}[theorem]{Lemma}

\newtheorem{claim}[theorem]{Claim}
\theoremstyle{definition}
\newtheorem{definition}[theorem]{Definition}
\newtheorem{remark}[theorem]{Remark}

\newtheorem{conjecture}[theorem]{Conjecture}

\newcommand{\F}{\mathbb{F}}

\newcommand{\R}{\mathbb{R}}
\newcommand{\Z}{\mathbb{Z}}

\newcommand{\calS}{\mathcal{S}}

\newcommand{\zo}{\{0, 1\}}

\newcommand{\eps}{\epsilon}

\newcommand{\wt}{\mathrm{wt}}
\newcommand{\CL}{\mathrm{CL}}

\newcommand{\Supp}{\mathrm{Supp}}

\newcommand{\supp}{\operatorname{supp}}

\newcommand{\one}{\mathbf 1}
\newcommand{\NRD}{\operatorname{NRD}}

\title{Multiplicative error set system sparsification:\\A simpler proof via chain length contraction}

\author{Joshua Brakensiek\thanks{Department of Electrical Engineering and Computer Sciences, University of California, Berkeley. Supported in part by the Simons Investigator award of Venkatesan Guruswami and NSF grants CCF-2211972 and DMS-2503280.} \and Venkatesan Guruswami\thanks{Simons Institute for the Theory of Computing and the University of California, Berkeley. Supported in part by
a Simons Investigator award and NSF award CCF-2211972.} \and Aaron Putterman\thanks{School of Engineering and Applied Sciences, Harvard University. Supported in part by the Simons Investigator Awards of Madhu Sudan and Salil Vadhan and AFOSR award FA9550-25-1-0112.}}

\date{\today}

\begin{document}

\maketitle

\begin{abstract}
The chain length of a set family $\mathcal{S} \subseteq 2^{[m]}$ is the largest ascending sequence of sets in containment order in the union-closure of $\mathcal S$. 
In this work, we provide a significantly simpler and more optimal characterization of the sparsifiability of set systems in terms of their chain length, improving on the work of Brakensiek and Guruswami [STOC 2025]. Our proof relies on a generalization of Karger's [SODA 1993] famous contraction algorithm and its recent linear algebraic extensions [Khanna-Putterman-Sudan SODA 2024], and our resulting bounds show that, just as VC dimension characterizes the \emph{additive sparsifiability} of a set system, chain length governs the \emph{multiplicative sparsifiability}. As a corollary, we obtain improved bounds for weighted CSP sparsification. 
\end{abstract}

\section{Introduction}

In theoretical computer science, sparsification describes a range of methods for reducing the storage size of a mathematical object while preserve some of its essential qualities. In this paper, we study the sparsification of set systems $\mathcal S \subseteq 2^{[m]}$ where $[m] := \{1, 2, \hdots, m\}$ is a finite set of atoms. Here, we define sparsification analogously to that of Karger's definition for graph cuts~\cite{DBLP:conf/soda/Karger93}, where we find a reweighting $w : [m] \to \R_{\ge 0}$ of the atoms such that for every set $S \in \mathcal S$, its weight before and after the reweighting match up to a multiplicative error of $1 \pm \eps$ (see Section~\ref{sec:prelim} for a formal definition). In general, the goal is to find a map $w$ such that the support $\supp(w) := \{i \in [m] : w(i) \neq 0\}$ is as small as possible.

This framework precisely recover the graph cut sparsification framework of Karger, if we let $[m]$ represent the set of edges of an undirected graph $G = (V, E)$, and let $\mathcal S$ be the family of all possible cuts of $G$ (i.e., sets of the form $(A \times (V \setminus A)) \cap E$, where $A \subseteq V$. In this paper, we prove general bounds on the sparsifiability of \emph{all} set systems, even when the initial atoms $[m]$ are also weighted.

The general study of set system sparsifiers was initiated by Brakensiek and Guruswami~\cite{brakensiek2025redundancy} to help resolve questions in the emerging area of \emph{CSP sparsification}, see Section~\ref{subsec:related-work} for further background. In particular, they showed that the optimal size of a (weighted) set system sparsifier is closely related to the \emph{chain length} (coined by Bessiere, Carbonnel, and Katsirelos~\cite{bessiere2020Chain}) of the set system. For a set system $\mathcal S \subseteq 2^{[m]}$, we define its chain length $\CL(\mathcal S)$ to be the length of the longest ascending chain of sets $S_1 \subsetneq \cdots \subsetneq S_{\CL(\mathcal S)}$ contained in $\bigcup \mathcal S$, the closure of $S$ with respect to set union--see Section~\ref{sec:prelim} for a formal definition. We now state the main result of Brakensiek and Guruswami~\cite{brakensiek2025redundancy} on weighted set system sparsification.

\begin{theorem}[Theorem~1.4 of \cite{brakensiek2025redundancy}]\label{thm:BG}
     For a set system $\mathcal S \subseteq 2^{[m]}$, weights $w: [m] \rightarrow \R_{\geq 0}$ and $\eps > 0$, there exists a $(1 \pm \eps)$ sparsifier of $\mathcal S$ which retains only 
     \[
     O \left ( \frac{ \CL(\mathcal S) \log^6 m}{\eps^2} \right )
     \]
     many atoms. 
\end{theorem}

Of note, this bound is ``near-optimal'' in the sense that for any set system $\mathcal S \subseteq 2^{[m]}$, there is a set of weights $w$ such that any sparsifier needs at least $\CL(\mathcal S)$ atoms (see Lemma~8.9 of \cite{brakensiek2025redundancy}\footnote{Lemma~8.9 appears in the full version of their paper at \url{https://arxiv.org/abs/2411.03451}.}).  
Furthermore, by a data structure lower bound of Carlson, Kolla, Srivastava and Trevisan~\cite{carlson2019optimal} for graph cuts, we know that the dependence on $\eps > 0$ is optimal in the worst case. The main inefficiency in the bound of Theorem~\ref{thm:BG} is thus the additional factor of $\log^6(m)$.

The proof of Theorem~\ref{thm:BG} is rather complex, as the proof first seeks to find near-optimal sparsifiers in the unweighted setting by adapting techniques pioneered by Gilmer~\cite{gilmer2022constant} in his recent breakthrough on the union-closed set conjecture (see \cite{brakensiek2025redundancy} for a much more detailed history of this problem).

\subsection{Main Result}

Our main result is a much more direct proof of Theorem~\ref{thm:BG} using a ``contraction''-style argument similar to those used in many papers in the sparsification literature~\cite{DBLP:conf/soda/Karger93,BK96,FHH11,KK15, ghaffari2017random, KPS24}. As a consequence, we also get a much sharper asymptotic analysis.

\begin{theorem}[Main Result, see Theorem~\ref{thm:main}]\label{thm:main-intro}
For a set system $\mathcal S \subseteq 2^{[m]}$, weights $w: [m] \rightarrow \R_{\geq 0}$ and $\eps > 0$, there exists a $(1 \pm \eps)$ sparsifier of $\mathcal S$ which retains only 
     \[
     O \left ( \frac{ \CL(\mathcal S) \cdot \log^2(\CL(\mathcal S)/\eps) \cdot (\log\log(\CL(\mathcal S)/\eps))^2}{\eps^2} \right )
     \]
     many atoms.
\end{theorem}

Theorem~\ref{thm:main-intro} improves on Theorem~\ref{thm:BG} in two ways. First, the multiplicative overhead of $\log^6 m$ is reduced to a much smaller $\log^{2+o(1)}(\CL(\mathcal S)/\eps)$. Second, the bound in Theorem~\ref{thm:main-intro} is independent\footnote{We observe that the bound in Theorem~\ref{thm:BG} could also be made independent of $m$ by applying a recursive argument similar to that of Section~\ref{subsec:dimension-free}. Even so, the asymptotics of Theorem~\ref{thm:main-intro} would still be superior.} of $m$!

As an application of the above theorem, one can consider a graph $G = (V, E)$, and the set system $\mathcal{S} \subseteq 2^E$ which contains all the \emph{cuts} in the graph $G$. I.e., for every set $T \subseteq V$, one can consider the set of edges $E_T \subseteq E$ which is \emph{cut} by the set $T$ and let $\mathcal{S} = \bigcup_{T \subseteq V} E_T$. Sparsifying the set system $\mathcal{S}$ is thus \emph{equivalent} to creating a \emph{cut-sparsifier} of the graph $G$ in the classical sense of \cite{BK96}. Because the chain length of the set of cuts is bounded by $|V|$, \cref{thm:main-intro} recovers (up to an extra $\widetilde{O}(\log(m))$ factor) the sparsifier size bound of \cite{BK96}. 

Note that \cref{thm:main-intro} is however \emph{not} efficiently implementable. This is in large part due to the fact that, given a set system $\mathcal{S} \subseteq 2^{[m]}$, there is no known efficient algorithm for computing $\mathrm{CL}(\calS)$.

\subsubsection{Improvement to Code Sparsification}

As an additional corollary to \cref{thm:main-intro}, we give the first bound 
for linear code sparsification (see~\cite{KPS24}) which is \emph{independent} of the underlying finite field. In more detail, given a finite field $\F_q$ on $q$ elements, we define a linear code to be an $n$-dimensional subspace $C \subseteq \F_q^m$. A $(1 \pm \eps)$ sparsifier of $C$ is a reweighting $w : [m] \to \R_{\ge 0}$ such that for every $c \in C$ we have that
\[
    \sum_{i=1}^m w(i) \one[c_i \neq 0] \in (1 \pm \eps) \sum_{i=1}^m \one[c_i \neq 0].
\]
A main result of Khanna, Putterman, and Sudan~\cite{KPS24} (see also \cite{khanna2025efficient}) is that $C$ has a $(1\pm \eps)$ sparsifier of size\footnote{Size here is meant to be to the number of non-zero weights that are assigned; equivalently, the number of coordinates in $[m]$ that are retained.}  $|\supp(w)| = n \log^{O(1)}(n) \log q / \eps^2$. If we capture linear code sparsification by letting our set system $\mathcal S$ be $\{\supp c : c \in C\}$, then we obtain a linear code sparsifier of size $\CL(\mathcal S) \log^{2+o(1)}(\CL(\mathcal S)/\eps) / \eps^2$. From existing results on chain length~\cite{bessiere2020Chain,brakensiek2025redundancy}, we know that $\CL(\mathcal S)$ is precisely the dimension $n$ of the code $C$. 
This gives the first field-size independent bound of $n \cdot \log^{2+o(1)}(n/\eps) / \eps^2$ for code sparsification! 

Lastly, recall that both \cite{brakensiek2025redundancy}, \cref{thm:BG} and \cite{KPS24} used \cref{thm:BG} and linear code sparsification respectively to design \emph{CSP sparsifiers}. Because of our better parameters in \cref{thm:main-intro}, we immediately obtain tighter bounds for these applications.

\subsection{Related Work}\label{subsec:related-work}

We now briefly discuss how our work connects to other parts of the sparsification and broader TCS literature.

\subsubsection{Code, CSP, and Cayley graph sparsification} As previously mentioned, Khanna, Putterman and Sudan~\cite{KPS24} pioneered the concept of \emph{code sparsification}. In their original paper, \cite{KPS24} presented multiple applications of code sparsification including \emph{constraint satisfaction problem (CSP) sparsification} and \emph{Cayley graph sparsification}.

CSP sparsification was introduced by Kogan and Krauthgamer~\cite{KK15} to generalize graph cut sparsification to broader families of discrete structures, including hypergraph cut sparsification (e.g., \cite{soma2019spectral, chen2020near, kapralov2021towards, kapralov2022spectral, Lee23, jambulapati2023chaining, khanna2024near}). A more systematic investigation of general CSP sparsification was started by Filtser--Krauthgamer~\cite{FK17} and Butti--{\v Z}ivn{\'y}~\cite{BZ20}. However, the case of linear equations over a finite field was not resolved until the method of code sparsification was developed~\cite{KPS24}. In a follow-up Khanna, Putterman and Sudan~\cite{khanna2025efficient} extended these methods to equations over Abelian groups among other CSPs. However, results concerning all CSPs were not established until Brakensiek and Guruswami~\cite{brakensiek2025redundancy} generalized linear code sparsification to \emph{non-linear} codes (i.e., arbitrary set systems).

In a separate direction, the work of Khanna, Putterman and Sudan~\cite{KPS24} introduced \emph{Cayley graph sparsification} as a special case of graph cut sparsification where the input graph is a Cayley graph (i.e., the graph is defined by the generators of a suitable group). Here, the goal is not to construct any sparsifier, but rather a sparsifier which is itself a Cayley graph. Khanna, Putterman and Sudan~\cite{KPS24} handled the case in which the Cayley graph corresponds to spanning vectors of a finite vector space. More recently, such methods were generalized to arbitrary groups \cite{hsieh2026sparsifying}, and even a more general theory of sparsifying sums of PSD matrices \cite{basu2026sparsifying}.

\subsubsection{Chain length and query complexity} The notion of chain length was introduced by Bessiere, Carbonnel, and Katsirelos~\cite{bessiere2020Chain} in the context of analyzing the query complexity of constraint satisfaction problems in a model introduced by Bessiere et al.~\cite{bessiere2013constraint}. In this query model, one is not given explicit access to the constraints; rather, one queries a partial assignment to a subset of the variables, and the response is `YES' or `NO' depending on whether the partial assignment is consistent with all constraints induced by those variables. To connect this problem with the concept of chain length, consider a set system $\mathcal S \subseteq 2^{[m]}$ where $[m]$ represents the clauses of the CSP and $S \in \mathcal S$ if and only if there is a valid assignment to the CSP which satisfies the clauses indexed by $S$ (and no others). Then, $\widetilde{O}(\CL(\mathcal S))$ is an upper bound on the query complexity of this problem.
However, it is an open question whether this chain length bound is tight. See \cite{brakensiek2025redundancy} for further discussion.

\subsubsection{Connections to Learning Theory} Recall that chain length of $\mathcal S \subseteq 2^{[m]}$ is the length of the longest ascending chain in $\bigcup \mathcal S$. Another natural property of a set family $\bigcup \mathcal F$ is its \emph{VC-dimension}. That is, the size of the largest set $A \subseteq [m]$ such that for all $B \subseteq A$, there is some $S \in \mathcal F$ with $S \cap A = B$. Bessiere, Carbonnel, and Katsirelos~\cite{bessiere2020Chain} define the \emph{non-redundancy} of $\mathcal S$, denoted by $\NRD(\mathcal S)$ to be precisely the VC-dimension of $\bigcup \mathcal S$. Brakensiek and Guruswami~\cite{brakensiek2025redundancy} identify the non-redundancy of $\mathcal S$ is closely related to the sparsifiability of $\mathcal S$, assuming $\mathcal S$ is unweighted.

More broadly, the VC-dimension of a set family  $\mathcal F \subseteq 2^{[m]}$ is well-known to characterize the sample complexity needed to obtain a small \emph{additive-error} approximation of the set-sizes in the family. 
That is, with constant probability a random subset $T \subset [m]$ of size $O(\text{VC-dim}(\mathcal F)/\eps^2)$~\cite{VC71,AnthonyBartlett1999} 
satisfies $|T \cap S| \in |S| \pm \eps m$ 
for every $S \in \mathcal F$. Conversely, no set of size smaller than $\Omega_\eps(\text{VC-dim}(\mathcal F))$ can satisfy such a guarantee. 

In this view, the chain length characterizes the sample complexity needed to obtain \emph{multiplicative-error} approximation to set-sizes, in the weighted setting. Note that, unlike the VC-theorem which works with unweighted samples, reweighting of the chosen samples is necessary to obtain a multiplicative approximation. We also remark that an unweighted multiplicative sampling theorem (relative to VC-dimension) was proved by Li, Long, and Srinivasan~\cite{LiLongSrinivasan2001}; however, their notion of multiplicative approximation significantly differs from ours.

\subsection*{Organization}

In Section~\ref{sec:prelim}, we state some known results (Chernoff bounds, etc.). In Section~\ref{sec:contraction}, we prove a key counting bound by relating the technique of contractions to chain length. In Section~\ref{sec:sparsifiers}, we use the results of Section~\ref{sec:contraction} to construct a series of sparsifiers, culminating in Theorem~\ref{thm:main-intro}. In Section~\ref{sec:conclusion}, we give some concluding thoughts and open questions.

\section{Preliminaries}\label{sec:prelim}

We now present some background material on sparsification and related concepts.

\subsection{Notation and Sparsification Definitions}

For the purposes of our analysis, we represent set systems $\mathcal S \subseteq 2^{[m]}$ as a set of vectors (also called a \emph{code}) $C \subseteq \zo^m$, where $S \in \mathcal S$ if and only if $c \in C$ where $c \in \{0,1\}^m$ is defined by
\[
    c_i = \begin{cases}
    1 & i \in S\\
    0 & i \not\in S.
    \end{cases}
\]
Optionally, the coordinates $[m]$ of the code may be given \emph{weights} $w: [m] \rightarrow \R_{\geq 0}$. Note that when no weights are provided for the code, we assume that the weights are all $1$ (equivalently, an unweighted code). Occasionally, for a code $C \subseteq \zo^m$ and a set $S \subseteq [m]$, we will use the notation $C|_S = \{ c|_S: c\in C\}$ to be the coordinate restriction of all codewords in $C$ to the set $S$.

Our goal, given a parameter $\eps \in (0,1)$ is to design a \emph{sparsifier}:

\begin{definition}
For a code $C \subseteq \zo^m$, weights $w: [m] \rightarrow \R_{\geq 0}$, and an accuracy parameter $\eps > 0$, a $(1 \pm \eps)$ \emph{sparsifier} of $C$ is a new set of weights $\widetilde{w}: [m] \rightarrow \R_{\geq 0}$ such that, for every $c \in C$:
\[
\langle \widetilde{w}, c \rangle = \sum_{i = 1}^m \widetilde{w}(i) c_i \in (1 \pm \eps ) \langle w, c \rangle = \sum_{i = 1}^m w(i) c_i.
\]
The goal is to design sparsifiers which \emph{reduce} the support size of the starting code; i.e., which minimize $|\supp(\widetilde{w})|$.
\end{definition}

\begin{remark}
    Note that in some contexts, for an unweighted code $C \subseteq \zo^m$, we will use $w \cdot C$ to refer to the code which assigns weight $w$ to \emph{all} coordinates in $[m]$. Similarly, for codes $C_1 \subseteq \zo^{A_1}$, $C_2 \subseteq \zo^{A_2}$ with $A_1 \cap A_2 = \emptyset$, we may use $w_1 \cdot C_1 \cup w_2 \cdot C_2$ to refer to the code which assigns weights $w_1$ to coordinates in $A_1$ and $w_2$ to coordinates in $A_2$.
\end{remark}

For reference, we include the definitions of non-redundancy and chain length below. Note that non-redundancy essentially captures the largest \emph{diagonal} matrix (after permuting the codewords) which one can find in a set of $\zo$-valued vectors.

\begin{definition}
    Let $C \subseteq \{0,1\}^m$ be an arbitrary set of vectors. We say that the \emph{non-redundancy} of $C$ (denoted $\mathrm{NRD}(C)$) is the largest size of a set $S \subseteq [m]$ such that for every $j \in S$, there is a codeword $c \in C$ where $c_j = 1$, but for every other $i \in S \setminus \{j\}$, $c_i = 0$.
\end{definition}

Likewise, we use the notion of chain length, which is essentially the largest upper triangular submatrix contained in a set of vectors:

\begin{definition}
    Let $C \subseteq \zo^m$ be a code. A chain of length $\ell$ is a pair of injective maps $a: [\ell] \rightarrow m$ and $c: [\ell] \rightarrow C$ such that the following conditions hold:
    \begin{enumerate}
        \item $\forall i \in [\ell]: c(i)_{a(i)} = 1$
        \item $\forall 1 \leq i < j \leq \ell: c(i)_{a(j)} = 0$.
    \end{enumerate}
    The chain length of $C$, denoted by $\mathrm{CL}(C)$ is the length of the longest chain.
\end{definition}

Note that $\mathrm{NRD}(C) \leq \CL(C)$ as any diagonal matrix is trivially an upper triangular matrix. With this, we can make use of the following bound (observed in \cite{brakensiek2025redundancy}):

\begin{claim}\label{clm:CLVCdim}
    $|C| \leq (m+1)^{\mathrm{NRD}(C)} \leq (m+1)^{\mathrm{CL}(C)}$.
\end{claim}

\subsection{Concentration Bound}

We will make use of the following concentration bound when sub-sampling to construct our sparsifiers:

\begin{claim}\label{clm:concentrationBound}{\rm (\cite{FHH11})}
    Let $X_1, \dots X_{\ell}$ be random variables such that $X_i$ takes on value $1 / p_i$ with probability $p_i$, and is $0$ otherwise. Also, suppose that $\min_i p_i \geq p$. Then, with probability at least $1 - 2e^{-0.38 \eps^2 \ell p}$,
    \[
    \sum_i X_i \in (1 \pm \eps) \ell.
    \]
\end{claim}

\section{Chain Length Counting Bounds via Contractions}\label{sec:contraction}

In this section we show the following lemma via a simple contraction argument:

\begin{lemma}\label{lem:CLDecomposition}
    Let $C \subseteq \zo^m$ be any code. Then, for any parameter $d>0$, there is a set $T \subseteq [m]$ of size $|T| \leq \mathrm{CL}(C) \cdot d$ such that for any $\alpha \in \Z^+$, the number of codewords in $C|_{\bar{T}}$ of weight $\leq \alpha d$ is at most $\binom{\mathrm{CL}(C)}{\alpha} \cdot (m+1)^{\alpha}$.
\end{lemma}

\subsection{Contractions}

To start, we have the following claim which governs how the chain length of a set of vectors behaves under \emph{contractions}.

\begin{claim}
Let $C \subseteq \{0,1\}^m$ be an arbitrary set of vectors and let $i \in [m]$ be any coordinate such that there exists a $c \in C$ such that $c_i = 1$. Then, for $C' = \{c \in C: c_i = 0 \}$, we have that $\mathrm{CL}(C') \leq \mathrm{CL}(C) -1$.
\end{claim}

\begin{proof}
For the code $C'$, let the maps witnessing the chain length of $C'$ be denoted by $a', c'$, and let the chain length be $\ell'$. Now, let us define $a: [\ell' + 1] \rightarrow [m], c: [\ell'+1] \rightarrow C$, such that for $j \in [\ell'], c(j) = c'(j), a(j) = a'(j)$, and $a(\ell'+1) = i$ and $c(\ell'+1) = v$, where $v \in C$ is any codeword such that $v_i \neq 0$.

Observe that we trivially have $\forall j \in [\ell'+1]: c(j)_{a(j)} = 1$. Likewise, because every codeword $c$ in $C'$ satisfies $c_i = 0$ and the first $\ell'$ codewords mapped to by $c$ are all in $C'$, we also have that for any $j \in [\ell']$, $c(j)_{a(\ell'+1)} = c(j)_{i} = 0$. Together with the conditions already guaranteed by $c', a'$, this yields a chain of length $\ell'+1$ in $C$, and thus the claim. 
\end{proof}

\subsection{Contraction Algorithm}

Leveraging the above claim which governs chain length under contractions, we define the following iterative contraction procedure:

\begin{algorithm}[H]
    \caption{Contract$(C, \alpha)$}
    \While{$\mathrm{CL}(C) \geq \alpha$}{
    Choose $i$ uniformly from $\{ i \in [m]: \exists c \in C: c_i \neq 0\}$. \\
    Let $C = C - \{ c \in C: c_i = 1\}$.
    }
    Let $c$ be a random codeword in $C$. \\
    \Return{$c$.}
\end{algorithm}

Now, let us define a new quantity:

\begin{definition}
    The density of a code $C \subseteq \zo^m$ is given by 
    \[
    \Phi(C) = \min_{C' \subseteq C} \frac{|\mathrm{Supp}(C')|}{\mathrm{CL}(C')},
    \]
    where $\Supp(C') = \{i \in [m]: \exists c \in C': c_i \neq 0 \}$.
\end{definition}

We immediately get the following counting bound. 

\begin{claim}\label{clm:countingBound}
    Let $C \subseteq \zo^m$ be a code. Then, for any positive integer $\alpha$, the number of codewords of weight $\leq \alpha \cdot \Phi(C)$ is $\leq \binom{\mathrm{CL}(C)}{\alpha} \cdot (m+1)^{\alpha}$.
\end{claim}

\begin{proof}
    Fix any codeword $c$ of weight $\leq \alpha \Phi(C)$, and run $\mathrm{Contract}(C, \alpha)$. For some intermediate code $C'$ achieved during the contraction procedure, observe that if $\mathrm{CL}(C') = k$, then necessarily, $|\Supp(C')| \geq \Phi(C) \cdot k$. Thus, in the next random choice of $i$, the probability that our codeword $c$ satisfies $c_i = 1$ is at most \[
    \frac{\wt(c)}{\Phi(C) \cdot k} \leq \frac{\alpha \Phi(C)}{\Phi(C) \cdot k} \leq \frac{\alpha}{k},
    \]
    and thus $c$ survives the contraction with probability $\geq 1 - \frac{\alpha}{k}$. The probability that $c$ survives all contractions until $\mathrm{CL}(C) \leq \alpha$ is thus
    \[
    \geq \prod_{k = \alpha+1}^{\mathrm{CL}(C)} \left ( 1 - \frac{\alpha}{k} \right ) = \binom{\mathrm{CL}(C)}{\alpha}^{-1}.
    \]

    Conditioned on surviving, the codeword $c$ is then chosen uniformly at random among all surviving codewords, and thus returned by $\mathrm{Contract}(C, \alpha)$ with probability $\geq(m+1)^{-\alpha}$, where we are using \cref{clm:CLVCdim} which bounds the number of codewords in a code with $\CL$ at most $\alpha$. So, $c$ is returned by $\mathrm{Contract}(C, \alpha)$ with probability at least $(m+1)^{-\alpha} \cdot \binom{\mathrm{CL(C)}}{\alpha}^{-1}$. By taking the reciprocal, this yields the bound on the number of codewords of weight $\leq \alpha \cdot \Phi(C)$.
\end{proof}

Importantly, we also have the following claim:

\begin{claim}\label{clm:chainlengthdecrease}
    Let $C \subseteq \zo^m$ be a code, and let $C' \subseteq C$ such that $\mathrm{CL}(C') = \ell$. Let $T = \Supp(C')$, and let $C|_{\bar{T}}$ denote the code $C$ with the coordinates corresponding to the support of $C'$ removed. We have:
    \[
    \mathrm{CL}(C|_{\bar{T}}) \leq \mathrm{CL}(C) - \ell.
    \]
\end{claim}

\begin{proof}
    As before, let $a_1, c_1$ be the functions which witness the chain length of $C|_{\bar{T}}$, and let $a_2, c_2$ be the functions which witness the chain length of $\mathrm{CL}(C')$. Note that $\mathrm{Im}(a_1) \subseteq [m] - T \subseteq [m]$ and $\mathrm{Im}(a_2) \subseteq T \subseteq [m]$.

    We define functions $a, c$ as follows: $a: [\ell + \mathrm{CL}(C|_{\bar{T}})] \rightarrow [m]$, where $a[j], c[j] = a_2[j], c_2[j]$ if $j \leq \ell$, and otherwise $a[j], c[j] = a_1[j - \ell], c_1[j - \ell]$. 
Importantly, because for $j \in [\ell]$, $c(j) \in C'$, and $a(\geq \ell) \subseteq \bar{T}$, we have that $c(j)_{a(\geq \ell)} = 0$. The remaining conditions for chain length are then trivially satisfied by the fact that $a_1, c_1, a_2, c_2$ already satisfied the definition of chain length individually. 

    Pictorially, we have:
    \[
    C = \begin{bmatrix} C' & A \\
        0 & C|_{\bar{T}}
    \end{bmatrix},
    \]
    thus we can trivially compose the upper triangular matrices corresponding to $C', C|_{\bar{T}}$. 
\end{proof}

\begin{lemma}
    Let $C \subseteq \zo^m$ be any code. Then, for any parameter $d>0$, there is a set $T \subseteq [m]$ of size $|T| \leq \mathrm{CL}(C) \cdot d$ such that for any $\alpha \in \Z^+$, the number of codewords in $C|_{\bar{T}}$ of weight $\leq \alpha d$ is at most $\binom{\mathrm{CL}(C)}{\alpha} \cdot (m+1)^{\alpha}$.
\end{lemma}

\begin{proof}
    If $\Phi(C) > d$, then we are immediately done by invoking \cref{clm:countingBound}. Otherwise, there is some $C' \subseteq C$ for which 
    \begin{align}\label{eq:density}
    \frac{|\mathrm{Supp}(C')|}{\mathrm{CL}(C')} \leq d.
    \end{align}
    So, let $T = \mathrm{Supp}(C')$, and let us set $C = C|_{\bar{T}}$. Again, we can check whether $\Phi(C) > d$, if so, we are done, and we simply return the set $T$ as is. Otherwise, there is a new subcode $C'' \subseteq C$ for which $\frac{|\mathrm{Supp}(C'')|}{\mathrm{CL}(C'')} \leq d$. Again, we let $T \leftarrow T \cup \mathrm{Supp}(C'')$, let $C = C|_{\bar{T}}$ and continue on recursively. 

    Eventually, $\Phi(C) > d$ (or the entire support of the code is removed) at which point we terminate. All that remains is to bound the size of the set $T$. For this, every time we remove the support of a subcode $C'$, we increase the size of the set $T$ by at most $d \cdot \mathrm{CL}(C')$, as $|\Supp(C')| \leq d \cdot \mathrm{CL}(C')$ (by \cref{eq:density}). However, by \cref{clm:chainlengthdecrease}, every time we remove $\Supp(C')$ from $C$, $\mathrm{CL}(C)$ decreases by $\mathrm{CL}(C')$. Hence, the total number of coordinates removed in the set $T$ can be at most $d \cdot \mathrm{CL}(C)$, before the chain length of the remaining code goes to $0$. This yields the lemma.
\end{proof}

\section{Leveraging Decomposition to Build Sparsifiers}\label{sec:sparsifiers}

\subsection{Basic Sparsifier Construction for Unweighted Codes}

A consequence of \cref{lem:CLDecomposition} is that it allows us to build sparsifiers for unweighted codes which preserve only $\widetilde{O}(\mathrm{CL}(C))$ many coordinates:

\begin{theorem}
    Let $C \subseteq \zo^m$ be an unweighted code. Then, for any $\eps > 0$, there exists a $(1 \pm \eps)$ sparsifier of $C$ which preserves only 
    \[
    O(\CL(C) (\log m)^2 (\log\log m )^2  / \eps^2)
    \]
    many re-weighted coordinates.
\end{theorem}

The algorithm for achieving this sparsification is simple, and builds off of the intuition of \cite{KPS24, khanna2025efficient}. Namely, starting with a code $C \subseteq \zo^m$, we invoke the decomposition of \cref{lem:CLDecomposition} with parameter $d = \sqrt{\frac{m}{\CL(C)}}$. This yields two codes: $C_{\mathrm{Peel}}$ which contains the $\leq \CL(C) \cdot d \leq \sqrt{m \CL(C)}$ many coordinates that \cref{lem:CLDecomposition} peels off, along with $C_{\mathrm{remaining}}$, which contains all remaining coordinates, with the promise that in $C_{\mathrm{remaining}}$, for any $\alpha \in \Z^+$, the number of codewords in $(C_{\mathrm{remaining}})|_{\bar{T}}$ of weight $\leq \alpha d$ is at most $\binom{\mathrm{CL}(C)}{\alpha} \cdot (m+1)^{\alpha}$. We then sub-sample the coordinates of $C_{\mathrm{remaining}}$, and then recurse. The formal algorithm is presented below. 

\begin{algorithm}
    \caption{Sparsify$(C, \eps, \mathrm{counter}, m)$}\label{alg:sparsify}
    Let $\eta = \frac{1000 \log(m)}{(\eps / 20 \log\log(m))^2}$ 
    Let $d = \sqrt{\frac{m \eta}{\CL(C)}} $. \\
    \If{$\mathrm{counter} = \log\log(m)$}{
    \Return{$C$.}
    }
    Let $T$ be a set of coordinates as promised by \cref{lem:CLDecomposition} for $C$ with parameter $d$. \\
    Let $C_{\mathrm{Remaining}} = C|_{\bar{T}}$ and let $C_{\mathrm{Peel}} = C|_T$. \\
    Let $w \in \R$, $\widetilde{T} \subseteq \bar{T}$ be such that $|\widetilde{T}| \leq 2 \sqrt{\CL(C) \cdot m \eta} $, and $w \cdot C_{\widetilde{T}}$ is a $(1 \pm \eps / (20 \log\log(m))$ code sparsifier of $C_{\mathrm{Remaining}}$. \label{line:subsample}\\
    \Return{$w \cdot \mathrm{Sparsify}(C_{\widetilde{T}}, \eps, \mathrm{counter}+1, m) \cup \mathrm{Sparsify}(C_{\mathrm{Peel}}, \eps, \mathrm{counter}+1, m)$.}
\end{algorithm}

\smallskip
We now prove some basic claims regarding the above algorithm:

\begin{claim}
    When invoked on a code $C \subseteq \zo^m$, \cref{line:subsample} in \cref{alg:sparsify} is always possible. 
\end{claim}

\begin{proof}
    After performing the decomposition of \cref{lem:CLDecomposition} with parameter $d = \sqrt{\frac{m\cdot \eta}{\CL(C)}} $, it must be the case that in $C_{\mathrm{remaining}}$, for any $\alpha \in \Z^+$, the number of codewords in $(C_{\mathrm{remaining}})$ of weight $\leq \alpha d$ is at most $\binom{\mathrm{CL}(C)}{\alpha} \cdot (m+1)^{\alpha}$. Thus, if one randomly samples the coordinates of $C_{\mathrm{remaining}}$ at rate $p = \sqrt{\frac{\eta \CL(C)}{m}}$, and assigns weight $1/p$ to the sampled coordinates a simple Chernoff and union bound will show that every codeword's weight is preserved to a $(1 \pm \eps / 20 \log\log(m))$ factor with probability $1 - 1 / \mathrm{poly}(m)$. Indeed, for a fixed value of $\alpha$, a codeword of weight $\in [\alpha d, 2 \alpha d]$ will have its weight preserved to within a $(1 \pm \eps / 20 \log\log(m))$ factor with probability (using \cref{clm:concentrationBound}) at least
    \begin{align*}
    1 - 2e^{-0.38 (\eps / 20 \log\log(m))^2 \cdot \alpha d \cdot \sqrt{\frac{\eta \CL(C)}{m}}} &\geq 1 - 2e^{-0.38 (\eps / 20 \log\log(m))^2 \cdot \alpha \eta}\\
    &\geq 1 - 2e^{-0.38 \cdot \alpha \cdot 1000 \log(m)}\\
    &\geq 1 - \frac{1}{m^{100 \alpha}}.
    \end{align*}
    Then, we can take a union bound over all codewords of weight $\leq 2 \alpha d$, at most $\binom{\mathrm{CL}(C)}{2\alpha} \cdot (m+1)^{2\alpha} \leq (m+1)^{4 \alpha}$ of them, 
    and then a union bound over the at most $m$ choices of $\alpha$ to conclude that every codeword has its weight preserved to a $(1\pm \eps / (20 \log\log(m)))$ factor with probability $1 - \frac{1}{m^{95}}$ (in particular such a sparsification \emph{exists}).

    At the same time, a Chernoff bound will show that when sampling at this rate, with probability $\Omega(1)$, at most $2 \cdot \sqrt{\CL(C) \cdot \eta \cdot m}$ many coordinates will survive the sampling, and thus there exists a set of coordinates $|\widetilde{T}| \leq 2 \sqrt{\CL(C) \cdot \eta\cdot m} $, and weight $w \in \R$ such that $w \cdot C_{\widetilde{T}}$ is a $(1 \pm \eps / (20 \log\log(n))$ code sparsifier of $C_{\mathrm{Remaining}}$.
\end{proof}

\begin{claim}\label{clm:oneRoundAccurate}
    When invoked on a code $C \subseteq \zo^m$, $w \cdot C_{\widetilde{T}} \cup C_{\mathrm{Peel}}$ as produced in \cref{line:subsample} is a $(1 \pm \eps / 20\log\log(m))$ code sparsifier of $C$.
\end{claim}

\begin{proof}
    This follows because $C_{\bar{T}} \cup C_{\mathrm{Peel}}$ is a perfect sparsifier of $C$. Then, when replacing $C_{\bar{T}}$ with a $(1 \pm \eps / 20\log\log(m))$ sparsifier of $C_{\bar{T}}$, it follows by composition that the resulting code is a $(1 \pm \eps / 20\log\log(m))$ code sparsifier of $C$.
\end{proof}

\begin{claim}\label{clm:accuracy}
    Let $C \subseteq \zo^m$ be a code and let $\eps > 0$. Then $\mathrm{Sparsify}(C, \eps, 0)$ returns a code $\widetilde{C}$ which is a $(1 \pm \eps)$ sparsifier of $C$.
\end{claim}

\begin{proof}
    This follows by induction over the value of $\mathrm{Counter}$. Indeed, when sparsifying the code $C$, we consider all the resulting codes for which $\mathrm{Sparsify}$ is recursively called. We let $w_{i, 1} \cdot C^{(i)}_1, \dots w_{i, 2^i} \cdot C^{(i)}_{2^i}$ denote all $2^i$ codes for which the $\mathrm{Sparsify}$ function is called with $\mathrm{Counter} = i$. We inductively claim that $w_{i, 1} \cdot C^{(i)}_1 \cup  \dots \cup w_{i, 2^i} \cdot C^{(i)}_{2^i}$ is a $(1 \pm 3 i \eps / 20 \log\log(m))$ sparsifier of $C$.

    The base case follows trivially, using \cref{clm:oneRoundAccurate}. 
    
    The inductive case follows by observing that for each code $w_{i, j} \cdot C^{(i)}_j$, by invoking \cref{clm:oneRoundAccurate}, we replace $C^{(i)}_j$ by two codes $w_{i+1,2j } C^{(i+1)}_{2j}$ and $w_{i+1,2j +1} \cdot C^{(i)}_{2j+1}$ such that 
    \[
    w_{i+1,2j } C^{(i+1)}_{2j} \cup w_{i+1,2j +1} \cdot C^{(i)}_{2j+1}
    \]
    is a $(1 \pm \eps/ 20 \log\log(m))$ sparsifier of $w_{i, j} \cdot C^{(i)}_j$. By composition over all the codes $w_{i, 1} \cdot C^{(i)}_1 \cup  \dots \cup w_{i, 2^i} \cdot C^{(i)}_{2^i}$, it follows then that 
    \[
    w_{i+1, 1} \cdot C^{(i+1)}_1 \cup  \dots \cup w_{i+1, 2^{i+1}} \cdot C^{(i+1)}_{2^{i+1}}
    \]
    will be a $(1 \pm \eps / 20 \log\log(m)) \cdot (1 \pm 3\eps i / 20 \log\log(m)) \in (1 \pm 3\eps (i+1) / 20 \log\log(m))$-sparsifier.

    Taking $i = \log\log(m)$ yields our desired claim. 
\end{proof}

\begin{claim}\label{clm:space}
    Let $C \subseteq \zo^m$ be a code and let $\eps > 0$. Then $\mathrm{Sparsify}(C, \eps, 0)$ returns a code $\widetilde{C}$ which retains only $O \left ( \frac{\CL(C) \log^2(m) (\log\log(m))^2 }{\eps^2}\right )$ many coordinates.
\end{claim}

\begin{proof}
    We claim that whenever $\mathrm{Sparsify}$ is invoked on a code $C'$ with $\mathrm{Counter} = i$, then 
    \[
    |\Supp(C')| \leq 4 \cdot \CL(C) \cdot \left ( \frac{m}{\CL(C)} \right )^{1 / 2^i} \cdot \eta.
    \]
    The base case follows trivially: indeed, in the first invocation ($i=1$) of the algorithm, $|\Supp(C_{\mathrm{Peel}})| \leq \CL(C) \cdot \sqrt{\frac{m\cdot \eta}{\CL(C)}} $ which trivially satisfies the above bound, and $|\Supp(C_{\widetilde{T}})| \leq 2 \sqrt{\CL(C) \cdot m \cdot \eta} \leq 2 \cdot \CL(C) \cdot \sqrt{\frac{m}{\CL(C)}} \cdot \eta$, which also trivially satisfies the above bound. Now, we assume the claim holds by induction.

    Then, at level $i$ of the algorithm 
    \[
    |\Supp(C')| \leq 4 \CL(C) \cdot \left ( \frac{m}{\CL(C)} \right )^{1 / 2^i} \cdot \eta.
    \]
    When we sparsify $C'$, we obtain two codes, $C'_{\mathrm{Peel}}$ and $C'_{\widetilde{T'}}$. We have that (by our choice of $d$ in \cref{alg:sparsify})
    \[
    |\Supp(C'_{\mathrm{Peel}})| \leq \CL(C) \cdot \sqrt{\frac{4\CL(C) \cdot \left ( \frac{m}{\CL(C)} \right )^{1 / 2^i} \cdot \eta \cdot \eta}{\CL(C)}} \leq 4 \CL(C) \cdot \eta \cdot \left ( \frac{m}{\CL(C)} \right )^{1 / 2^{i+1}},
    \]
    as we desire. Likewise, 
    \[
    |\Supp(C'_{\widetilde{T'}})| \leq 2 \sqrt{\CL(C) \cdot 4\CL(C) \cdot \left ( \frac{m}{\CL(C)} \right )^{1 / 2^i} \cdot \eta \cdot \eta} \leq 4 \CL(C) \cdot \eta \cdot \left ( \frac{m}{\CL(C)} \right )^{1 / 2^{i+1}}
    \]
    \[
    \leq 4\CL(C) \cdot \eta \cdot \left ( \frac{m}{\CL(C)} \right )^{1 / 2^{i+1}},
    \]
    as we desire. 

    Now, once $i = \log\log(m)$, we have that 
    \[
    |\Supp(C')| \leq 4\CL(C) \cdot \left ( \frac{m}{\CL(C)} \right )^{1 / 2^{\log\log(m)}} \cdot \eta \leq 8 \CL(C) \cdot \eta.
    \]

    Thus, the returned code $\widetilde{C}$ is the union of $\leq 2^{\log\log(m)} = \log(m)$ many codes, each of which retains $\leq 8 \CL(C) \cdot \eta$ many coordinates. Thus, the total sparsifier size is bounded by \[
    \log(m) \cdot 8\CL(C) \cdot \eta = O \left ( \frac{\CL(C) \log^2(m) (\log\log(m))^2 }{\eps^2}\right ),
    \]
    as we desire.
\end{proof}

Together, these above claims give the following theorem:

\begin{theorem}\label{thm:CLSparsifier}
    For a code $C \subseteq \zo^m$ and $\eps > 0$, there exists a $(1 \pm \eps)$ code sparsifier of $C$ which retains only $O \left ( \frac{\CL(C) \log^2(m) (\log\log(m))^2 }{\eps^2}\right )$ many coordinates.
\end{theorem}

\begin{proof}
    We invoke \cref{alg:sparsify}. The accuracy of the sparsifier follows from \cref{clm:accuracy} and the size follows from \cref{clm:space}.
\end{proof}

Note that at this point, one can use \cref{thm:CLSparsifier} along with weighted to unweighted sparsifier reduction frameworks as in \cite{KPS24, khanna2025efficient, brakensiek2025redundancy} to deduce sparsifiers for \emph{weighted codes} which preserve $O \left ( \frac{\CL(C) \log^2(m) (\log\log(m))^2 }{\eps^2}\right )$ many coordinates.

\begin{corollary}\label{cor:weightedCLSparsifier}
        For a code $C \subseteq \zo^m$, weights $w: [m] \rightarrow \R_{\geq 0}$ and $\eps > 0$, there exists a $(1 \pm \eps)$ code sparsifier of $C$ which retains only $O \left ( \frac{\CL(C) \log^2(m) (\log\log(m))^2 }{\eps^2}\right )$ many coordinates.
\end{corollary}

For completeness, we include a proof of Corollary~\ref{cor:weightedCLSparsifier} in Appendix~\ref{appendix}.

\subsection{Dimension-Free Bounds}\label{subsec:dimension-free}

In this section, we show that the dependence on $\log(m)$ can be \emph{removed}, and instead replaced with dependence only on $\log(\CL)$, thus constituting a dimension-free sparsification result. To do this, we merely repeatedly apply the sparsification result of \cref{cor:weightedCLSparsifier}:

\begin{theorem}\label{thm:main}
     For a code $C \subseteq \zo^m$, weights $w: [m] \rightarrow \R_{\geq 0}$ and $\eps > 0$, there exists a $(1 \pm \eps)$ code sparsifier of $C$ which retains only 
     \[
     O \left ( \frac{ \CL(C) \cdot \log^2(\CL(C)/\eps) \cdot \log\log(\CL(C)/\eps)^2}{\eps^2} \right )
     \]
     many coordinates.
\end{theorem}

\begin{proof}
    First, let us use $K$ to denote the hidden constant in the $O(\cdot)$ notation above. In the below, we assume that $m$ is at least a sufficiently large constant to begin with. Now, we consider two cases: 
   
    \begin{enumerate}
        \item When $m \geq 2^{\CL(C) / (\eps^2)}$, then after applying \cref{cor:weightedCLSparsifier} with $\eps ' = \eps / Q \log m$ (for a large constant $Q$), we obtain a $( 1 \pm \eps')$ sparsifier of $C$ with support $\leq \frac{K \cdot \CL(C) \cdot \log^4(m) \cdot \log\log(m)^2}{\eps^2} \leq \log^6(m)$. 
        
        \item When $m < 2^{\CL(C) / (\eps^2)}$, we sparsify $C$ once with parameter $\eps' = \eps / Q \log(m)$, obtaining a code $C'$ which is a $(1 \pm \eps')$ sparsifier of $C$ with support $\leq \frac{K \cdot \CL(C) \cdot \log^4(m) \cdot \log\log(m)^2}{\eps^2} \leq K \cdot \CL(C)^{7} / \eps^{14}$. Then, we sparsify $C'$ with parameter $\eps / 2$, and obtain a code $C''$ which is a $(1 \pm \eps/2)$ sparsifier of $C'$, and retains at most
        \[
      \frac{K \cdot \CL(C) \cdot \log^2(\CL(C)/\eps) \cdot \log\log(\CL(C)/\eps)^2}{\eps^2}
        \]
        many coordinates.
    \end{enumerate}

    Thus, given an arbitrary codes $C$, we apply case (1) $\ell = O(\log^*(m))$ many times, yielding code $C^{(1)}, \dots C^{(\ell)}$. Then, we apply case (2) to $C^{(\ell)}$, yielding a code $C''$ which retains only 
    \[
     \leq \frac{K \cdot \CL(C) \cdot \log^2(\CL(C)/\eps) \cdot \log\log(\CL(C)/\eps)^2}{\eps^2}
    \]
    many coordinates.

    All that remains is to show that $C''$ is a $(1 \pm \eps)$ sparsifier of $C$. This follows by composition of the sparsifier accuracy: we let $m^{(i)}$ denote the support size of the $i$th sparsifier we create in the above chain (with $m^{(0)}$ being the starting value of $m$). Then, $C^{(1)}$ is a $(1 \pm \eps/ Q \log(m^{(0)}))$ sparsifier of $C$, $C^{(2)}$ is a $(1 \pm \eps/ \log(m^{(1)}))$ sparsifier of $C^{(1)}$, and so on. In particular, this implies that $C^{(\ell)}$ is a 
    \[
    \prod_{i = 0}^{\ell-1} (1 \pm \eps / Q \log(m^{(i)})) 
    \]
    sparsifier of $C$. Thus $C^{(\ell)}$ is a 
    \[
    \left (1 \pm \eps \cdot O\left (\sum_{i = 0}^{\ell-1} \frac{1}{Q \cdot \log(m^{(i)})} \right ) \right ) = \left (1 \pm \eps /10 \right )
    \]
    sparsifier of $C$, where we have chosen $Q$ to be a large enough constant. In particular, here we are using the fact that, letting $q = \log(m^{(\ell-1)})$,
    \[
    \sum_{i = 0}^{\ell-1} \frac{1}{\log(m^{(i)})} \leq \frac{1}{q} + \frac{1}{2^{q^{1/6}}} + \frac{1}{2^{\left (2^{q^{1/6}} \right )^{1/6}}} + \dots = O(1 / q) = O(1).
    \]
    This relation follows from the fact that $m^{(i)} \leq \log^6(m^{(i-1)})$ as established above. 
    
    Now $C'$ is a $(1 \pm \eps / Q)$ sparsifier of $C^{(\ell)}$, which implies that $C'$ is a $(1 \pm \eps /10) \cdot (1 \pm \eps/Q) \in (1 \pm \eps/4)$ sparsifier of $C$.

    Finally, $C''$ is a $(1 \pm \eps /2 )$ sparsifier of $C'$, and thus by composition, $C''$ is a $(1 \pm \eps)$ sparsifier of $C$, as we desire (note, here we are using that $\eps$ is a sufficiently small constant).
\end{proof}

\section{Conclusion}\label{sec:conclusion}

In this paper, we constructed  sparsifiers for arbitrary set systems of near-optimal size. In particular, we simplify the proof strategy of Brakensiek and Guruswami~\cite{brakensiek2025redundancy} by presenting a contraction-based counting bound which results in an overall simpler proof in addition to better asymptotics. Furthermore, by recursively applying our sparsifier, our bound only depends on the accuracy $\eps$ and the chain length of the underlying set system. We conclude this paper with an exciting open direction and our conjecture regarding it.

\textbf{Truly linear sparsifiers.} A celebrated result of Batson, Spielman and Srivastava~\cite{BSS09} states that graphs on $n$ vertices have $(1 \pm \eps)$ cut sparsifiers with $O(n/\eps^2)$-edges. Furthermore, such asymptotics are known to be tight up to a constant factor~\cite{carlson2019optimal}. Proving an analogue of \cite{BSS09} for arbitrary set systems seems quite difficult as there is no clear spectral analogue of set system sparsification (however, see \cite{KPS25spectral}). Currently, the most general extensions of \cite{BSS09} are to the regime of \emph{PSD matrix sparsification} in the work of \cite{silva2015sparse} (which even includes arity $3$ \emph{hypergraph} cut sparsification).
Nonetheless, we boldly conjecture that analogous constructions should exist in general.

\begin{conjecture}\label{conj:CL-linear}
     For a set system $\mathcal S \subseteq 2^{[m]}$, weights $w: [m] \rightarrow \R_{\geq 0}$ and $\eps > 0$, there exists a $(1 \pm \eps)$ sparsifier of $\mathcal S$ which retains only 
     \[
     O \left ( \frac{ \CL(\mathcal S)}{\eps^2} \right )
     \]
     many coordinates.
\end{conjecture}

Of note, even proving Conjecture~\ref{conj:CL-linear} for linear code sparsification would be a significant result.

\section*{Acknowledgments}

We thank anonymous reviewers for many helpful comments improving the presentation of this paper. We also thank an anonymous reviewer of \cite{brakensiek2025redundancy} whose comments partially inspired \cref{thm:main}. A. P. thanks Sanjeev Khanna and Madhu Sudan for helpful conversations. 

\bibliographystyle{alpha}
\bibliography{references}

\appendix

\section{Proof of Corollary~\ref{cor:weightedCLSparsifier}}\label{appendix}

The specific method we use for proving Corollary~\ref{cor:weightedCLSparsifier} closely follows the methodology in Section 8 of \cite{brakensiek2025redundancy} (cf. Theorem~8.4). To start, we prove a version of this corollary in the setting where all the weights are bounded:

\begin{corollary}\label{cor:BoundedWeightedCLSparsifier}
        For a code $C \subseteq \zo^m$, weights $w: [m] \rightarrow \R_{\geq 0}$ such that $w(i) \leq m^3$ and $\eps > 0$, there exists a $(1 \pm \eps)$ code sparsifier of $C$ which retains only $O \left ( \frac{\CL(C) \log^2(m) (\log\log(m))^2 }{\eps^2}\right )$ many coordinates.
\end{corollary}

\begin{proof}
    We adapt the proof of Lemma~8.15 in \cite{brakensiek2025redundancy}. First, if $\eps \leq 1 / \sqrt{m}$, then we simply return all the coordinates of the code $C$, as $m \leq 1 / \eps^2$. So, for the rest of the proof, we assume that $\eps > 1 / \sqrt{m}$. We will also assume that $\min_i w(i) = 1$, as otherwise we can simply rescale the weights by $\min_i w(i)$ to ensure this is the case (if the min weight is $0$, we can simply delete these coordinates).

    The key intuition is now that the weights are bounded, we will simply duplicate the coordinates of the code a number of times proportional to their weight, thus transforming the weighted code into an unweighted code. With this unweighted code, we can then invoke \cref{thm:CLSparsifier}.

    To make this unweighted code for each $i \in [m]$, we define $b(i) = \lfloor \frac{2w(i)}{\eps}\rfloor \leq 2m^4$ to be the number of times we duplicate coordinate $i$. Thus, we create a new code $\widetilde{C} \subseteq \zo^{\widetilde{m}}$ where $\widetilde{m} = \sum_{i = 1}^m b(i)$ where coordinate $i$ of $C$ is replaced with $b(i)$ unweighted copies in $\widetilde{C}$. Now, we claim that for any codeword $c \in C$, its corresponding version $\widetilde{c}$ satisfies $  \frac{\eps}{2} \cdot  \sum_{i = 1}^{\widetilde{m}} \widetilde{c}_i \in (1 \pm \eps / 2) \cdot \sum_{i = 1}^m w(i) c_i$. Indeed, this follows because
    \[
    \frac{\eps}{2} \cdot \sum_{i = 1}^{\widetilde{m}} \widetilde{c}_i = \sum_{i = 1}^{m} \frac{\eps}{2} \cdot \Bigl\lfloor \frac{2w(i)}{\eps}\Bigr\rfloor \cdot c_i \in \sum_{i = 1}^{m} \frac{\eps}{2} \cdot \left[\frac{2w(i)}{\eps}-1, \frac{2w(i)}{\eps} \right ] \cdot c_i
    \]
    \[
    \in \sum_{i = 1}^{m} [1 - \eps/2, 1] \cdot w_i \cdot c_i \in (1 \pm \eps / 2) \cdot \sum_{i = 1}^m w(i) c_i.
    \]

    Thus, to conclude the proof, we must only construct this now \emph{unweighted} code $\widetilde{C}$, sparsify it to accuracy $(1 \pm \eps/3)$ using \cref{thm:CLSparsifier} to create a new, weighted code $\widehat{C}$, and then return $\frac{\eps}{2} \cdot \widehat{C}$. By the composition of sparsifier accuracy, this will be a $(1 \pm \eps/2) (1 \pm \eps/3)$ sparsifier of $C$, which constitutes a $(1 \pm \eps)$ sparsifier. The number of coordinates retained in the sparsifier is bounded by $O \left ( \frac{\CL(C) \log^2(m) (\log\log(m))^2 }{\eps^2}\right )$, as we are sparsifying a code with $\widetilde{m} = \mathrm{poly}(m)$ many coordinates.
\end{proof}

With this corollary established, we can now proceed to creating sparsifiers for codes with arbitrary weights.

\begin{proof}[Proof of \cref{cor:weightedCLSparsifier}]
We fix the code $C \subseteq \{0,1\}^m$, the weights $w : [m] \to \R_{> 0}$, and $\eps \in (0,1)$. We assume WLOG that $\CL(C) \geq 1$ and that $\eps \geq  8 / \sqrt{m}$ (otherwise, the code $C$ already satisfies the desired size bound). 

We group the coordinates $[m]$ according to their weights. Borrowing notation from the proof of Theorem~8.4 of \cite{brakensiek2025redundancy}, we define a function $t : [m] \to \Z$ such that
\[
    t(i) := \left\lfloor \frac{\log w(i)}{3\log m}\right\rfloor,
\]
and for each coordinate $i \in [m]$, we assign its group to be $t(i)$. For all $t \in \mathbb Z$, we let $I_t \subseteq [n]$ denote the set of all coordinates in group $t$. For each codeword $c \in C$, we then say that the \emph{type} of the codeword is the maximum group for any of its nonzero coordinates, i.e., 
\[
\mathrm{type}(c) := \max_{i \in \supp(c)} t(i).
\]
For all $t \in \mathbb Z$, we let $C_t$ denote the set of all $c \in C$ with type $t$.

The key observation, as used in \cite{KPS24, khanna2025efficient, brakensiek2025redundancy} is that one must only consider sparsifying a codeword $c \in C_t$ with respect to the coordinates $I_t$; this is because the vast majority of its weight is contained in these coordinates. Furthermore, within the set of coordinates $I_t$, all weights are related by a factor of at most $m^3$.

Define $t \in \mathbb Z$ to be \emph{proper} if there exists $z \in (C_{t} \cup C_{t+1})|_{I_t}$ which has at least one coordinate equal to $1$. For proper $t \in \mathbb Z$, let $\widetilde{w}_t : [m] \to \R_{\ge 0}$ be an $(\eps/2)$-sparsifier of $(C_{t} \cup C_{t+1})|_{I_t} \cup \{\mathbf{1}[\supp(I_t)]\}$. Note that $\CL(C_{t} \cup C_{t+1})|_{I_t} \cup \{\mathbf{1}[\supp(I_t)]\}) \ge 1$, so 
\begin{align*}
\CL((C_{t} \cup C_{t+1})|_{I_t} \cup \{\mathbf{1}[\supp(I_t)]\}) &\le \CL(C_{t} \cup C_{t+1})|_{I_t} \cup \{\mathbf{1}[\supp(I_t)]\}) +1\\&\le 2\CL(C_{t} \cup C_{t+1})|_{I_t} \cup \{\mathbf{1}[\supp(I_t)]\}).
\end{align*}
Thus, by \cref{cor:BoundedWeightedCLSparsifier}, we can ensure that
\begin{align}
    \left | {\supp(\widetilde{w}_t)}  \right | \leq O \left (\CL((C_{t} \cup C_{t+1})|_{I_t})\log^2(m)\log\log(m)^2/\eps^2 \right ).\label{eq:w-t}
\end{align}

Let $\widetilde{w} : [m] \to \R_{\ge 0}$ be the sum of these sparsifiers $\{\widetilde{w}_t : t \in \mathbb Z, t \text{ proper}\}$. This is well-defined as any coordinate appears in exactly one $I_t$ and thus can be nonzero for only $\widetilde{w}_t$. It remains only to show that $\widetilde{w}$ is indeed a $( 1\pm \eps)$ sparsifier for $C$ with weights $w$ and that it satisfies our desired sparsity bound. Adapting a computation of \cite{brakensiek2025redundancy}, we have for all $c \in C$ with $t := \mathrm{type}(c)$ that
\begin{align*}
\langle \widetilde{w}, c\rangle &= \langle \widetilde{w}_{t-1}, c|_{I_{t-1}}\rangle + \langle \widetilde{w}_t, c|_{I_t}\rangle +\sum_{t'\le t-2} \langle \widetilde{w}_{t'}, c|_{I_{t'}}\rangle\\
&\in [1-\eps/2,1+\eps/2] \cdot \sum_{i \in I_{t-1} \cup I_t} w(i)c_i + \sum_{t'\le t-2} \langle \widetilde{w}_{t'}, c|_{I_{t'}}\rangle\\
&= [1-\eps/2, 1+\eps/2] \cdot \left[\langle w, c\rangle - \sum_{i \in \supp(c) \setminus (I_{t-1} \cup I_t)} w(i)\right] + \sum_{t'\le t-2} \langle \widetilde{w}_{t'}, c|_{I_{t'}}\rangle.
\end{align*}

First, we argue that $\sum_{t'\le t-2} \langle \widetilde{w}_{t'}, c|_{I_{t'}}\rangle$ is small. For this, suppose for the sake of contradiction that for some $t' \leq t-2$, there is a choice of $i \in I_{t'}$ such that $\widetilde{w}(i) \geq 2m^{3t'+4}$. Then, observe that for the codeword $\mathbf{1}[\supp(I_{t'})]$, its weight reported by the sparsifier for $I_{t'}$ is at least $2m^{3t'+4}$, which is $\geq 2 \cdot \sum_{i \in I_{t'}} w(i)$, as each $w(i) \leq m^{3t' + 3}$ (and $|I_{t'}| \leq m$). Thus, for this codeword $\mathbf{1}[\supp(I_{t'})]$, $\widetilde{w}$ restricted to $I_{t'}$ (i.e,. $\widetilde{w}_{t'}$) \emph{would not} have been a $(1 \pm \eps)$ sparsifier. So, we see that $\widetilde{w}(i) \leq 2m^{3t'+4}$ for each $i \in I_{t'}$, and thus \[
\sum_{t'\le t-2} \langle \widetilde{w}_{t'}, c|_{I_{t'}}\rangle \leq m \cdot 2m^{3t'+4} \leq 2m^{3t'+5} \leq 2m^{3t-1}.
\]

Since $t = \mathrm{type}(c)$, we have that $\langle w, c\rangle \ge m^{3t}$. This immediately implies that $\sum_{t'\le t-2} \langle \widetilde{w}_{t'}, c|_{I_{t'}}\rangle \leq \frac{2}{m} \cdot \langle w, c\rangle \leq \frac{\eps}{4} \cdot \langle w, c\rangle$ since $\eps \ge 8/\sqrt{m}$.

Further, for all $i \in \supp(c) \setminus (I_{t-1} \cup I_t)$, we have that $w(i) < m^{3t-3}$. Therefore,  the total contribution of $w(i)$ for $i \in \supp(c) \setminus (I_{t-1} \cup I_t)$ is at most $m \cdot m^{3t-3} \le \frac{\eps}{4} \langle w, c\rangle$.
So, 
\begin{align*}
\langle \widetilde{w}, c\rangle &\in [1-\eps/2, 1+\eps/2] \cdot [1-\eps/4, 1] \cdot \langle w, c\rangle  + [0, \eps/4] \cdot \langle w, c\rangle\subseteq [1-\eps, 1+\eps] \langle w, c\rangle ,
\end{align*}
as desired. It thus suffices to bound $\left |{\supp(\widetilde{w})} \right |$. By (\cref{eq:w-t}), we have that
\begin{align}
\left |{\supp(\widetilde{w})} \right | \leq \sum_{t\text{ proper}} \left |{\supp(\widetilde{w}_t)} \right | \leq \sum_{t \in \mathbb Z} O \left ( \CL((C_{t} \cup C_{t+1})|_{I_t})\log^2(m)\log\log(m)^2/\eps^2 \right ).\label{eq:w-t2}
\end{align}
Analogous to Claim~8.16 of \cite{brakensiek2025redundancy}, we claim that 
\[
\sum_{t\in \mathbb Z} \CL((C_{t} \cup C_{t+1})|_{I_t}) \leq 2 \CL(C).
\]
To see this, let $A_1, \dots A_{\ell} \subseteq I_1, \dots I_{\ell}$ and $B_1, \dots B_{\ell} \subseteq (C_1 \cup C_2), (C_2 \cup C_3) \dots (C_{\ell}, C_{\ell+1})$ denote witnesses for the chain lengths of $\CL((C_{t} \cup C_{t+1})|_{I_t}): t \in [\ell]$. The key observation is that $A_{\mathrm{odd}} = A_1 \cup A_3 \cup A_5 \cup \dots , B_{\mathrm{odd}} = B_1 \cup B_3 \cup B_5 \cup \dots$ and $A_{\mathrm{even}}=A_2 \cup A_4 \cup A_6 \cup \dots , B_{\mathrm{even}} = B_2 \cup B_4 \cup B_6 \cup \dots$ are \emph{both} valid witnesses of the chain length of $C$. To see why, we observe that because $B_i \subseteq C_i \cup C_{i+1}$, it must be the case that every codeword $c \in B_i$ satisfies $c|_{I_{\geq i+2}} = 0$. Thus, for any coordinate $p \in A_{\geq i+2}$, we see that $c_p = 0$. Importantly then, because $A_i, B_i$ is a chain, and $A_{i+2}, B_{i+2}$ is defined on a disjoint support, their concatenation still remains a chain as per \cref{clm:chainlengthdecrease}.

With this, we then immediately have that 
\[
\sum_t \CL((C_{t} \cup C_{t+1})|_{I_t}) = \sum_{t: t = 0\!\!\!\!\mod 2} \CL((C_{t} \cup C_{t+1})|_{I_t}) + \sum_{t: t = 1\!\!\!\!\mod 2} \CL((C_{t} \cup C_{t+1})|_{I_t})
\]
\[
\leq \CL(C) + \CL(C) \leq 2 \CL(C),
\]
and thus by (\cref{eq:w-t2}) we can conclude that 
\[
\left |{\supp(\widetilde{w})} \right | \leq O \left ( \CL(C)\log^2(m)\log\log(m)^2/\eps^2 \right ),
\]
as desired.
\end{proof}

\end{document}